\documentclass[12pt]{amsart}
\usepackage{amsmath}
\usepackage{amsthm}
\usepackage{amssymb}
\usepackage{tikz}

\usepackage[all]{xy}
\usepackage[left=2.5cm,right=2.5cm,top=3cm,bottom=3cm]{geometry}
\usepackage{eucal}
\usepackage{palatino}

\usepackage{euler}
\usepackage{mathrsfs}
\usepackage{amssymb}
\usepackage{amscd}
\usepackage{latexsym}
\usepackage{epsfig}
\usepackage{graphicx}
\usepackage{amsfonts}
\usepackage{psfrag}
\usepackage{caption}
\usepackage{rotating}
\usepackage{mathtools}
\usepackage{fullpage}
\usepackage{etoolbox} 
\usepackage{fontenc}

\usepackage[normalem]{ulem}

\usepackage{pifont}
\usepackage{epsfig}

\usepackage{url}
\usepackage{cite}
\usepackage{dsfont}
\usepackage{array}

\newtheorem{theorem}{Theorem}
\newtheorem{proposition}{Proposition}

\newtheorem{lemma}{Lemma}

\theoremstyle{definition}
\newtheorem{example}{Example}
\newtheorem{remark}{Remark}

\DeclareMathOperator{\GL}{GL}

\DeclareMathOperator{\Adj}{Ad}

\DeclareMathOperator{\adj}{ad}

\DeclareMathOperator{\End}{End}

\DeclareMathOperator{\PSL}{PSL}

\newtoggle{showColor}
\newtoggle{showNotes}
\toggletrue{showColor}
\togglefalse{showNotes}

\iftoggle{showNotes} {%
	\newcommand{\note}[1]{{\textcolor{red}{$\langle$#1$\rangle$}}} 
}{%
	\newcommand{\note}[1]{}
}

\title{Properties of Nilpotent Orbit Complexification}
\author{Peter Crooks}
\address{Department of Mathematics, University of Toronto, Canada}
\email{~~~peter.crooks@utoronto.ca}

\begin{document}
\begin{abstract}
We consider aspects of the relationship between nilpotent orbits in a semisimple real Lie algebra $\mathfrak{g}$ and those in its complexification $\mathfrak{g}_{\mathbb{C}}$. In particular, we prove that two distinct real nilpotent orbits lying in the same complex orbit are incomparable in the closure order. Secondly, we characterize those $\mathfrak{g}$ having non-empty intersections with all nilpotent orbits in $\mathfrak{g}_{\mathbb{C}}$. Finally, for $\mathfrak{g}$ quasi-split, we characterize those complex nilpotent orbits containing real ones.   
\end{abstract}
\maketitle

\section{Introduction}
\subsection{Background and Statement of Results}
Real and complex nilpotent orbits have received considerable attention in the literature. The former have been studied in a variety of contexts, including differential geometry, symplectic geometry, and Hodge theory (see \cite{Schmid-Vilonen}). Also, there has been some interest in concrete descriptions of the poset structure on real nilpotent orbits in specific cases (see \cite{Djokovic3}, \cite{Djokovic2}). By contrast, complex nilpotent orbits are studied in algebraic geometry (see \cite{Brieskorn}, \cite{Kraft-Procesi}, \cite{Slodowy}) and representation theory --- in particular, Springer Theory (see \cite{Chriss-Ginzburg}). 

Attention has also been given to the interplay between real and complex nilpotent orbits, with the Kostant-Sekiguchi Correspondence (see \cite{Sekiguchi}) being perhaps the most famous instance. Accordingly, the present article provides additional points of comparison between real and complex nilpotent orbits. Specifically, let $\mathfrak{g}$ be a finite-dimensional semisimple real Lie algebra with complexification $\mathfrak{g}_{\mathbb{C}}$. Each real nilpotent orbit $\mathcal{O}\subseteq\mathfrak{g}$ lies in a unique complex nilpotent orbit $\mathcal{O}_{\mathbb{C}}\subseteq\mathfrak{g}_{\mathbb{C}}$, the complexification of $\mathcal{O}$. The following is our main result.

\begin{theorem}\label{Main Theorem}
The process of nilpotent orbit complexification has the following properties.
\begin{itemize}
\item[(i)] Every complex nilpotent orbit is realizable as the complexification of a real nilpotent orbit if and only if $\mathfrak{g}$ is quasi-split and has no simple summand of the form $\mathfrak{so}(2n+1,2n-1)$.
\item[(ii)] If $\mathfrak{g}$ is quasi-split, then a complex nilpotent orbit $\Theta\subseteq\mathfrak{g}_{\mathbb{C}}$ is realizable as the complexification of a real nilpotent orbit if and only if $\Theta$ is invariant under conjugation with respect to the real form $\mathfrak{g}\subseteq\mathfrak{g}_{\mathbb{C}}$.
\item[(iii)] If $\mathcal{O}_1,\mathcal{O}_2\subseteq\mathfrak{g}$ are real nilpotent orbits satisfying $(\mathcal{O}_1)_{\mathbb{C}}=(\mathcal{O}_2)_{\mathbb{C}}$, then either $\mathcal{O}_1=\mathcal{O}_2$ or these two orbits are incomparable in the closure order.  
\end{itemize}
\end{theorem} 

\subsection{Structure of the Article} We begin with an overview of nilpotent orbits in semisimple real and complex Lie algebras. In recognition of Theorem \ref{Main Theorem} (iii), and of the role played by the unique maximal complex nilpotent orbit $\Theta_{\text{reg}}(\mathfrak{g}_{\mathbb{C}})$ throughout the article, Section \ref{The Closure Orders} reviews the closure orders on the sets of real and complex nilpotent orbits. In Section \ref{Partitions}, we recall some of the details underlying the use of decorated partitions to index nilpotent orbits.

Section \ref{Nilpotent Orbit Complexification} is devoted to the proof of Theorem \ref{Main Theorem}. In Section \ref{The Comlexification Map}, we represent nilpotent orbit complexification as a poset map $\varphi_{\mathfrak{g}}$ between the collections of real and complex nilpotent orbits. Next, we show this map to have a convenient description in terms of decorated partitions. Section \ref{Surjectivity} then directly addresses the proof of Theorem \ref{Main Theorem} (i), formulated as a characterization of when $\varphi_{\mathfrak{g}}$ is surjective. Using Proposition \ref{reduction}, we reduce this exercise to one of characterizing surjectivity for $\mathfrak{g}$ simple. Together with the observation that surjectivity implies $\mathfrak{g}$ is quasi-split and is implied by $\mathfrak{g}$ being split, Proposition \ref{reduction} allows us to complete the proof of Theorem \ref{Main Theorem} (i).

We proceed to Section \ref{The Image}, which provides the proof of Theorem \ref{Main Theorem} (ii). The essential ingredient is Kottwitz's work \cite{Kottwitz}. We also include Proposition \ref{Unique Partition}, which gives an interesting sufficient condition for a complex nilpotent orbit to be in the image of $\varphi_{\mathfrak{g}}$.

In Section \ref{Fibres}, we give a proof of Theorem \ref{Main Theorem} (iii). Our proof makes extensive use of the Kostant-Sekiguchi Correspondence, the relevant parts of which are mentioned. 

\subsection*{Acknowledgements}  The author is grateful to John Scherk for discussions that prompted much of this work. The author also acknowledges Lisa Jeffrey and Steven Rayan for their considerable support. This work was partially funded by NSERC CGS and OGS awards. 

\section{Nilpotent Orbit Generalities}
\subsection{Nilpotent Orbits}\label{Notation and Conventions}
We begin by fixing some of the objects that will persist
throughout this article. Let $\mathfrak{g}$ be a finite-dimensional semisimple real Lie algebra with adjoint group $G$. Also, let $\mathfrak{g}_{\mathbb{C}}:=\mathfrak{g}\otimes_{\mathbb{R}}\mathbb{C}$ be the complexification of $\mathfrak{g}$, whose adjoint group is the complexification $G_{\mathbb{C}}$. One has the adjoint representations $$\Adj: G\rightarrow\GL(\mathfrak{g}) \text{ and } \Adj_{\mathbb{C}}:
G_{\mathbb{C}}\rightarrow\GL(\mathfrak{g}_{\mathbb{C}})$$ of $G$ and $G_{\mathbb{C}}$, respectively. Differentiation then gives the adjoint representations
of $\mathfrak{g}$ and $\mathfrak{g}_{\mathbb{C}}$, namely $$\adj:\mathfrak{g}\rightarrow\mathfrak{gl}(\mathfrak{g}) \text{ and } \adj_{\mathbb{C}}:\mathfrak{g}_{\mathbb{C}}\rightarrow\mathfrak{gl}(\mathfrak{g}_{\mathbb{C}}).$$ Recall that an element $\xi\in\mathfrak{g}$ (resp. $\xi\in\mathfrak{g}_{\mathbb{C}}$) is called nilpotent if $\adj(\xi):\mathfrak{g}\rightarrow\mathfrak{g}$ (resp. $\adj_{\mathbb{C}}(\xi):\mathfrak{g}_{\mathbb{C}}\rightarrow\mathfrak{g}_{\mathbb{C}}$) is a nilpotent
vector space endomorphism. The nilpotent cone $\mathcal{N}(\mathfrak{g})$ (resp. $\mathcal{N}(\mathfrak{g}_{\mathbb{C}})$) is then the subvariety
of nilpotent elements of $\mathfrak{g}$ (resp. $\mathfrak{g}_{\mathbb{C}}$). A \textit{real} (resp. \textit{complex}) \textit{nilpotent orbit} is an orbit of a
nilpotent element in $\mathfrak{g}$ (resp. $\mathfrak{g}_{\mathbb{C}}$) under the adjoint representation of $G$ (resp. $G_{\mathbb{C}}$). Since
the adjoint representation occurs by means of Lie algebra automorphisms, a real (resp. complex) nilpotent orbit is equivalently defined to be a $G$-orbit (resp. $G_{\mathbb{C}}$-orbit) in $\mathcal{N}(\mathfrak{g})$ (resp. $\mathcal{N}(\mathfrak{g}_{\mathbb{C}})$).
By virtue of being an orbit of a smooth $G$-action, each real nilpotent orbit is an immersed
submanifold of $\mathfrak{g}$. However, as $G_{\mathbb{C}}$ is a complex linear algebraic group, a complex nilpotent orbit is a smooth locally closed complex subvariety of $\mathfrak{g}_{\mathbb{C}}$.

\subsection{The Closure Orders}\label{The Closure Orders} The sets $\mathcal{N}(\mathfrak{g})/G$ and $\mathcal{N}(\mathfrak{g}_{\mathbb{C}})/G_{\mathbb{C}}$ of real and complex nilpotent orbits are finite and carry the so-called closure order. In both cases, this is a partial order defined by \begin{equation}\label{Closure Order}\mathcal{O}_1\leq\mathcal{O}_2 \text{ if and only if } \mathcal{O}_1\subseteq\overline{\mathcal{O}_2}.\end{equation} In the real case, one takes closures in the classical topology on $\mathfrak{g}$. For the complex case, note that a complex nilpotent orbit $\Theta$ is a constructible subset of $\mathfrak{g}_{\mathbb{C}}$, so that its Zariski and classical closures agree. Accordingly, $\overline{\Theta}$ shall denote this common closure.

\begin{example}\label{Special Linear} Suppose that $\mathfrak{g}_{\mathbb{C}}=\mathfrak{sl}_n(\mathbb{C})$, whose adjoint group is $G_{\mathbb{C}}=\PSL_n(\mathbb{C})$. The nilpotent elements of $\mathfrak{sl}_n(\mathbb{C})$ are precisely the nilpotent $n\times n$ matrices, so that the nilpotent $\PSL_n(\mathbb{C})$-orbits are exactly the ($\GL_n(\mathbb{C})$-) conjugacy classes of nilpotent matrices. The latter are indexed by the partitions of $n$ via Jordan canonical forms. Given a partition $\lambda=(\lambda_1,\lambda_2,\ldots,\lambda_k)$ of $n$, let $\Theta_{\lambda}$ be the $\PSL_n(\mathbb{C})$-orbit of the nilpotent matrix with Jordan blocks of sizes $\lambda_1,\lambda_2,\ldots,\lambda_k$, read from top-to-bottom. It is a classical result of Gerstenhaber \cite{Gerstenhaber} that $\Theta_{\lambda}\leq\Theta_{\mu}$ if and only if $\lambda\leq\mu$ in the dominance order. (See \cite{Stanley} for a precise definition of this order.) \hfill $\square$
\end{example}
The poset $\mathcal{N}(\mathfrak{g}_{\mathbb{C}})/G_{\mathbb{C}}$ has a unique maximal element $\Theta_{\text{reg}}(\mathfrak{g}_{\mathbb{C}})$, called the regular nilpotent orbit. It is the collection of all elements of $\mathfrak{g}_{\mathbb{C}}$ which are simultaneously regular and nilpotent. In the framework of Example \ref{Special Linear}, $\Theta_{\text{reg}}(\mathfrak{sl}_n(\mathbb{C}))$ corresponds to the partition $(n)$. 

\subsection{Partitions of Nilpotent Orbits}\label{Partitions}
Generalizing Example \ref{Special Linear}, it is often natural to associate a partition to each real and complex nilpotent orbit. One sometimes endows these partitions with certain decorations and then uses decorated partitions to enumerate nilpotent orbits. It will be advantageous for us to recall the construction of the underlying (undecorated) partitions. Our exposition will be largely based on Chapters 5 and 9 of \cite{Collingwood-McGovern}. 

Suppose that $\mathfrak{g}$ comes equipped with a faithful representation $\mathfrak{g}\subseteq\mathfrak{gl}(V)=\End_{\mathbb{F}}(V)$, where $V$ is a finite-dimensional vector space over $\mathbb{F}=\mathbb{R}$ or $\mathbb{C}.$\footnote{Since $\mathfrak{g}$ is semisimple, the adjoint representation is a canonical choice of faithful $V$. Nevertheless, it will be advantageous to allow for different choices.} The choice of $V$ determines an assignment of partitions to nilpotent orbits in both $\mathfrak{g}$ and $\mathfrak{g}_{\mathbb{C}}$. To this end, fix a real nilpotent orbit $\mathcal{O}\subseteq\mathcal{N}(\mathfrak{g})$ and choose a point $\xi\in\mathcal{O}$. We may include $\xi$ as the nilpositive element of an $\mathfrak{sl}_2(\mathbb{R})$-triple $(\xi,h,\eta)$, so that $$[\xi,\eta]=h,\text{ } [h,\xi]=2\xi,\text{ } [h,\eta]=-2\eta.$$ Regarding $V$ as an $\mathfrak{sl}_2(\mathbb{R})$-module, one has a decomposition into irreducibles, $$V=\bigoplus_{j=1}^kV_{\lambda_j},$$ where $V_{\lambda_j}$ denotes the irreducible $\lambda_j$-dimensional representation of $\mathfrak{sl}_2(\mathbb{R})$ over $\mathbb{F}$. Let us require that $\lambda_1\geq\lambda_2\geq\ldots\geq\lambda_k$, so that $(\lambda_1,\lambda_2,\ldots,\lambda_k)$ is a partition of $\dim_{\mathbb{F}}(V)$. Accordingly, we define the partition of $\mathcal{O}$ to be $$\lambda(\mathcal{O}):=(\lambda_1,\lambda_2,\ldots,\lambda_k).$$ It can be established that $\lambda(\mathcal{O})$ depends only on $\mathcal{O}$.

The faithful representation $V$ of $\mathfrak{g}$ canonically gives a faithful representation $\tilde{V}$ of $\mathfrak{g}_{\mathbb{C}}$. Indeed, if $V$ is over $\mathbb{C}$, then one has an inclusion $\mathfrak{g}_{\mathbb{C}}\subseteq\mathfrak{gl}(V)$ (so $\tilde{V}=V$). If $V$ is over $\mathbb{R}$, then the inclusion $\mathfrak{g}\subseteq\mathfrak{gl}(V)$ complexifies to give a faithful representation $\mathfrak{g}_{\mathbb{C}}\subseteq\mathfrak{gl}(V_{\mathbb{C}})$ (ie. $\tilde{V}=V_{\mathbb{C}}$). In either case, one proceeds in analogy with the real nilpotent case, using the faithful representation to yield a partition $\lambda(\Theta)$ of a complex nilpotent orbit $\Theta\subseteq\mathcal{N}(\mathfrak{g}_{\mathbb{C}})$. The only notable difference with the real case is that $\mathfrak{sl}_2(\mathbb{R})$ is replaced with $\mathfrak{sl}_2(\mathbb{C})$. 

\begin{example}\label{Complex Special Linear}
One can use the framework developed above to index the nilpotent orbits in $\mathfrak{sl}_n(\mathbb{C})$ using the partitions of $n$. This coincides with the indexing given in Example \ref{Special Linear}.\hfill $\square$
\end{example}

\begin{example}\label{Real Special Linear}
The nilpotent orbits in $\mathfrak{sl}_n(\mathbb{R})$ are indexed by the partitions of $n$, after one replaces certain partitions with decorated counterparts. Indeed, if $\lambda$ is a partition of $n$ having only even parts, we replace $\lambda$ with the decorated partitions $\lambda_+$ and $\lambda_{-}$. Otherwise, we leave $\lambda$ undecorated.\hfill $\square$ 
\end{example}

\begin{example}\label{Signed Young Diagram}
Suppose that $n\geq 3$ and consider $\mathfrak{g}=\mathfrak{su}(p,q)$ with $1\leq q\leq p$ and $p+q=n$. This Lie algebra is a real form of $\mathfrak{sl}_n(\mathbb{C})$. Now, let us regard a partition of $n$ as a Young diagram with $n$ boxes. Furthermore, recall that a signed Young diagram is a Young diagram whose boxes are marked with $+$ or $-$, such that the signs alternate across each row (for more details, see Chapter 9 of \cite{Collingwood-McGovern}). We restrict our attention to the signed Young diagrams of signature $(p,q)$, namely those for which $+$ and $-$ appear with respective multiplicities $p$ and $q$. It turns out that the nilpotent orbits in $\mathfrak{su}(p,q)$ are indexed by the signed Young diagrams of signature $(p,q)$.\hfill $\square$ 
\end{example}

\begin{example}\label{Complex Special Orthogonal}
Suppose that $\mathfrak{g}_{\mathbb{C}}=\mathfrak{so}_{2n}(\mathbb{C})$ with $n\geq 4$. Taking our faithful representation to be $\mathbb{C}^{2n}$, nilpotent orbits in $\mathfrak{so}_{2n}(\mathbb{C})$ are assigned partitions of $2n$. The partitions realized in this way are those in which each even part appears with even multiplicity. One extends these partitions to an indexing set by replacing each $\lambda$ having only even parts with the decorated partitions $\lambda_+$ and $\lambda_{-}$.\hfill $\square$    
\end{example}

\begin{example}\label{Indefinite Special Orthogonal}
Suppose that $n\geq 3$ and consider $\mathfrak{g}=\mathfrak{so}(p,q)$ with $1\leq q\leq p$ and $p+q=n$. Note that $\mathfrak{so}(p,q)$ is a real form of $\mathfrak{g}_{\mathbb{C}}=\mathfrak{so}_{n}(\mathbb{C})$. As with Example \ref{Signed Young Diagram}, we will identify partitions of $n$ with Young diagrams having $n$ boxes. We begin with the signed Young diagrams of signature $(p,q)$ such that each even-length row appears with even multiplicity and has its leftmost box marked with $+$. To obtain an indexing set for the nilpotent orbits in $\mathfrak{so}(p,q)$, we decorate two classes of these signed Young diagrams $Y$. Accordingly, if $Y$ has only even-length rows, then remove $Y$ and add the four decorated diagrams $Y_{+,+}$, $Y_{+,-}$, $Y_{-,+}$, and $Y_{-,-}$. Secondly, suppose that $Y$ has at least one odd-length row, and that each such row has an even number of boxes marked $+$, or that each such row has an even number of boxes marked $-$. In this case, we remove $Y$ and add the decorated diagrams $Y_+$ and $Y_-$.\hfill $\square$     
\end{example}

\section{Nilpotent Orbit Complexification}\label{Nilpotent Orbit Complexification}
\subsection{The Complexification Map}\label{The Comlexification Map}
There is a natural way in which a real nilpotent orbit determines a complex one. Indeed, the inclusion $\mathcal{N}(\mathfrak{g})\subseteq\mathcal{N}(\mathfrak{g}_{\mathbb{C}})$ gives rise to a map $$\varphi_{\mathfrak{g}}:\mathcal{N}(\mathfrak{g})/G\rightarrow\mathcal{N}(\mathfrak{g}_{\mathbb{C}})/G_{\mathbb{C}}$$ $$\mathcal{O}\mapsto\mathcal{O}_{\mathbb{C}}.$$
Concretely, $\mathcal{O}_{\mathbb{C}}$ is just the unique complex nilpotent orbit containing $\mathcal{O}$, and we shall call it the \textit{complexification} of $\mathcal{O}$. Let us then call $\varphi_{\mathfrak{g}}$ the \textit{complexification map} for $\mathfrak{g}$. 

It will be prudent to note that the process of nilpotent orbit complexification is well-behaved with respect to taking partitions. More explicitly, we have the following proposition.

\begin{proposition}\label{Complexification and Partitions}
Suppose that $\mathfrak{g}$ is endowed with a faithful representation $\mathfrak{g}\subseteq\mathfrak{gl}(V)$. If $\mathcal{O}$ is a real nilpotent orbit, then $\lambda(\mathcal{O}_{\mathbb{C}})=\lambda(\mathcal{O})$.
\end{proposition}

\begin{proof}
Choose a point $\xi\in\mathcal{O}$ and include it in an $\mathfrak{sl}_2(\mathbb{R})$-triple $(\xi,h,\eta)$ as in Section \ref{Partitions}. Note that $(\xi,h,\eta)$ is then additionally an $\mathfrak{sl}_2(\mathbb{C})$-triple in $\mathfrak{g}_{\mathbb{C}}$. Hence, we will prove that the faithful representation $\tilde{V}$ of $\mathfrak{g}_{\mathbb{C}}$ decomposes into irreducible $\mathfrak{sl}_2(\mathbb{C})$-representations according to the partition $\lambda(\mathcal{O})$. 

Let us write $\lambda(\mathcal{O})=(\lambda_1,\ldots,\lambda_k)$, so that \begin{equation}\label{decomposition}V=\bigoplus_{j=1}^kV_{\lambda_j}\end{equation} is the decomposition of $V$ into irreducible $\mathfrak{sl}_2(\mathbb{R})$-representations. If $V$ is over $\mathbb{C}$, then $\tilde{V}=V$ and \eqref{decomposition} is a decomposition of $\tilde{V}$ into irreducible $\mathfrak{sl}_2(\mathbb{C})$-representations. If $V$ is over $\mathbb{R}$, then $\tilde{V}=V_{\mathbb{C}}$ and $$V_{\mathbb{C}}=\bigoplus_{j=1}^k(V_{\lambda_j})_{\mathbb{C}}$$ is the decomposition of $\tilde{V}$ into irreducible representations of $\mathfrak{sl}_2(\mathbb{C})$. In each of these two cases, we have $\lambda(\mathcal{O}_{\mathbb{C}})=\lambda(\mathcal{O})$.     
\end{proof}

Proposition \ref{Complexification and Partitions} allows us to describe $\varphi_{\mathfrak{g}}$ in more combinatorial terms. To this end, fix a faithful representation $\mathfrak{g}\subseteq\mathfrak{gl}(V)$. As in Examples \ref{Complex Special Linear}--\ref{Indefinite Special Orthogonal}, we obtain index sets $I(\mathfrak{g})$ and $I(\mathfrak{g}_{\mathbb{C}})$ of decorated partitions for the real and complex nilpotent orbits, respectively. We may therefore regard $\varphi_{\mathfrak{g}}$ as a map $$\varphi_{\mathfrak{g}}:I(\mathfrak{g})\rightarrow I(\mathfrak{g}_{\mathbb{C}}).$$ Now, let $P(\mathfrak{g}_{\mathbb{C}})$ be the set of all partitions of the form $\lambda(\Theta)$, with $\Theta\subseteq\mathfrak{g}_{\mathbb{C}}$ a complex nilpotent orbit. One has the map $$I(\mathfrak{g}_{\mathbb{C}})\rightarrow P(\mathfrak{g}_{\mathbb{C}}),$$ sending a decorated partition to its underlying partition. Proposition \ref{Complexification and Partitions} is then the statement that the composite map $$I(\mathfrak{g})\xrightarrow{\varphi_{\mathfrak{g}}} I(\mathfrak{g}_{\mathbb{C}})\rightarrow P(\mathfrak{g}_{\mathbb{C}})$$ sends an index in $I(\mathfrak{g})$ to its underlying partition. Let us denote this composite map by $\psi_{\mathfrak{g}}:I(\mathfrak{g})\rightarrow P(\mathfrak{g}_{\mathbb{C}})$.

We will later give a characterization of those semisimple real Lie algebras $\mathfrak{g}$ for which $\varphi_{\mathfrak{g}}$ is surjective. To help motivate this, we investigate the matter of surjectivity in some concrete examples.

\begin{example}\label{Surjectivity: Special Linear}
Recall the parametrizations of nilpotent orbits in $\mathfrak{g}=\mathfrak{sl}_n(\mathbb{R})$ and $\mathfrak{g}_{\mathbb{C}}=\mathfrak{sl}_n(\mathbb{C})$ outlined in Examples \ref{Real Special Linear} and \ref{Complex Special Linear}, respectively. We see that $I(\mathfrak{g}_{\mathbb{C}})=P(\mathfrak{g}_{\mathbb{C}})$ and $\varphi_{\mathfrak{g}}=\psi_{\mathfrak{g}}$. The surjectivity of $\varphi_{\mathfrak{g}}$ then follows immediately from that of $\psi_{\mathfrak{g}}$.\hfill $\square$ 
\end{example}

\begin{example}\label{Surjectivity: Indefinite Special Unitary}
Let the nilpotent orbits in $\mathfrak{g}=\mathfrak{su}(n,n)$ be parametrized as in Example \ref{Signed Young Diagram}. We then have $\mathfrak{g}_{\mathbb{C}}=\mathfrak{sl}_{2n}(\mathbb{C})$, whose nilpotent orbits are indexed by the partitions of $2n$. Given such a partition $\lambda$, let $Y$ denote the corresponding Young diagram. Since $Y$ has an even number of boxes, it has an even number, $2k$, of odd-length rows. Label the leftmost box in $k$ of these rows with $+$, and label the leftmost box in each of the remaining $k$ rows with $-$. Now, complete this labelling to obtain a signed Young diagram $\tilde{Y}$, noting that $\tilde{Y}$ then has signature $(n,n)$. Hence, $\tilde{Y}$ corresponds to a nilpotent orbit in $\mathfrak{su}(n,n)$ and $\psi_{\mathfrak{g}}(\tilde{Y})=\lambda$. It follows that $\psi_{\mathfrak{g}}$ is surjective. Since $I(\mathfrak{g}_{\mathbb{C}})=P(\mathfrak{g}_{\mathbb{C}})$ and $\varphi_{\mathfrak{g}}=\psi_{\mathfrak{g}}$, we have shown $\varphi_{\mathfrak{g}}$ to be surjective. A similar argument establishes surjectivity when $\mathfrak{g}=\mathfrak{su}(n+1,n)$.\hfill $\square$ 
\end{example}

\begin{example}\label{Surjectivity: Indefinite Special Orthogonal}
Let us consider $\mathfrak{g}=\mathfrak{so}(2n+2,2n)$, with nilpotent orbits indexed as in Example \ref{Indefinite Special Orthogonal}. Noting Example \ref{Complex Special Orthogonal}, a partition $\lambda$ of $4n+2$ represents a nilpotent orbit in $\mathfrak{g}_{\mathbb{C}}=\mathfrak{so}_{4n+2}(\mathbb{C})$ if and only if each even part of $\lambda$ occurs with even multiplicity. Since $4n+2$ is even and not divisible by $4$, it follows that any such $\lambda$ has exactly $2k$ odd parts for some $k\geq 1$. Let $Y$ be the Young diagram corresponding to $\lambda$, and label the leftmost box in $k-1$ of the odd-length rows with +. Next, label the leftmost box in each of $k-1$ different odd-length rows with $-$. Finally, use $+$ to label the leftmost box in each of the two remaining odd-length rows. Let $\tilde{Y}$ be any completion of our labelling to a signed Young diagram, such that the leftmost box in each even-length row is marked with $+$. Note that $\tilde{Y}$ has signature $(2n+2,2n)$. It follows that $\tilde{Y}$ represents a nilpotent orbit in $\mathfrak{so}(2n+2,2n)$ and $\psi_{\mathfrak{g}}(\tilde{Y})=\lambda$. Furthermore, $I(\mathfrak{g}_{\mathbb{C}})=P(\mathfrak{g}_{\mathbb{C}})$ and $\varphi_{\mathfrak{g}}=\psi_{\mathfrak{g}}$, so that $\varphi_{\mathfrak{g}}$ is surjective. \hfill $\square$
\end{example}

\begin{example}\label{Failed Surjectivity: Indefinite Special Orthogonal}
Suppose that $\mathfrak{g}=\mathfrak{so}(2n+1,2n-1)$, whose nilpotent orbits are parametrized in Example \ref{Indefinite Special Orthogonal}. Let the nilpotent orbits in $\mathfrak{g}_{\mathbb{C}}=\mathfrak{so}_{4n}(\mathbb{C})$ be indexed as in Example \ref{Complex Special Orthogonal}. There exist partitions of $4n$ having only even parts, with each part appearing an even number of times. Let $\lambda$ be one such partition, which by Example \ref{Indefinite Special Orthogonal} represents a nilpotent orbit in $\mathfrak{so}_{4n}(\mathbb{C})$. Note that every signed Young diagram with underlying partition $\lambda$ must have signature $(2n,2n)$. In particular, $\lambda$ cannot be realized as the image under $\psi_{\mathfrak{g}}$ of a signed Young diagram indexing a nilpotent orbit in $\mathfrak{so}(2n+1,2n-1)$. It follows that $\psi_{\mathfrak{g}}$ and $\varphi_{\mathfrak{g}}$ are not surjective.\hfill $\square$
\end{example}

\subsection{Surjectivity}\label{Surjectivity}    
We now address the matter of classifying those semisimple real Lie algebras $\mathfrak{g}$ for which $\varphi_{\mathfrak{g}}$ is surjective. To proceed, we will require some additional machinery. Let $\mathfrak{p}\subseteq\mathfrak{g}$ be the $(-1)$-eigenspace of a Cartan involution, and let $\mathfrak{a}$ be a maximal abelian subspace of $\mathfrak{p}$. Also, let $\mathfrak{h}$ be a Cartan subalgebra of $\mathfrak{g}$ containing $\mathfrak{a}$, and choose a fundamental Weyl chamber $C\subseteq\mathfrak{h}$. Given a complex nilpotent orbit $\Theta\subseteq\mathfrak{g}_{\mathbb{C}}$, there exists an $\mathfrak{sl}_2(\mathbb{C})$-triple $(\xi,h,\eta)$ in $\mathfrak{g}_{\mathbb{C}}$ with the property that $\xi\in\Theta$ and $h\in C$. The element $h\in C$ is uniquely determined by this property, and is called the characteristic of $\Theta$. Theorem 1 of \cite{Djokovic1} then states that $\Theta\cap\mathfrak{g}\neq\emptyset$ if and only if $h\in\mathfrak{a}$. If $\mathfrak{g}$ is split, then $\mathfrak{a}=\mathfrak{h}$, and the following lemma is immediate.

\begin{lemma}\label{split}
If $\mathfrak{g}$ is split, then $\varphi_{\mathfrak{g}}$ is surjective.
\end{lemma}

Let us now consider necessary conditions for surjectivity. To this end, recall that $\mathfrak{g}$ is called \textit{quasi-split} if there exists a subalgebra $\mathfrak{b}\subseteq\mathfrak{g}$ such that $\mathfrak{b}_{\mathbb{C}}$ is a Borel subalgebra of $\mathfrak{g}_{\mathbb{C}}$. However, the following characterization of being quasi-split will be more suitable for our purposes.

\begin{lemma}\label{quasi-split}
The Lie algebra $\mathfrak{g}$ is quasi-split if and only if $\Theta_{\emph{reg}}(\mathfrak{g}_{\mathbb{C}})$ is in the image of $\varphi_{\mathfrak{g}}$. In particular, $\mathfrak{g}$ being quasi-split is a necessary condition for $\varphi_{\mathfrak{g}}$ to be surjective.
\end{lemma}

\begin{proof}
Proposition 5.1 of \cite{Rothschild} states that $\mathfrak{g}$ is quasi-split if and only if $\mathfrak{g}$ contains a regular nilpotent element of $\mathfrak{g}_{\mathbb{C}}$. Since $\Theta_{\text{reg}}(\mathfrak{g}_{\mathbb{C}})$ consists of all such elements, this is equivalent to having $\Theta_{\text{reg}}(\mathfrak{g}_{\mathbb{C}})\cap\mathfrak{g}\neq\emptyset$ hold. This latter condition holds precisely when $\Theta_{\text{reg}}(\mathfrak{g}_{\mathbb{C}})$ is in the image of $\varphi_{\mathfrak{g}}$.
\end{proof}

Lemmas \ref{split} and \ref{quasi-split} establish that $\varphi_{\mathfrak{g}}$ being surjective is a weaker condition than having $\mathfrak{g}$ be split, but stronger than having $\mathfrak{g}$ be quasi-split. Furthermore, since $\mathfrak{su}(n,n)$ is not a split real form of $\mathfrak{sl}_{2n}(\mathbb{C})$, Example \ref{Surjectivity: Indefinite Special Unitary} establishes that surjectivity is strictly weaker than $\mathfrak{g}$ being split. Yet, as $\mathfrak{so}(2n+1,2n-1)$ is a quasi-split real form of $\mathfrak{so}_{4n}(\mathbb{C})$, Example \ref{Failed Surjectivity: Indefinite Special Orthogonal} demonstrates that surjectivity is strictly stronger than having $\mathfrak{g}$ be quasi-split. To obtain a more precise measure of the strength of the surjectivity condition, we will require the following proposition.

\begin{proposition}\label{reduction}
Suppose that $\mathfrak{g}$ decomposes as a Lie algebra into $$\mathfrak{g}=\bigoplus_{j=1}^k\mathfrak{g}_j,$$ where $\mathfrak{g}_1,\ldots,\mathfrak{g}_k$ are simple real Lie algebras. Let $G_1,\ldots,G_k$ denote the respective adjoint groups.
\begin{itemize}
\item[(i)] The map $\varphi_{\mathfrak{g}}:\mathcal{N}(\mathfrak{g})/G\rightarrow\mathcal{N}(\mathfrak{g}_{\mathbb{C}})/G_{\mathbb{C}}$ is surjective if and only each orbit complexification map $\varphi_{\mathfrak{g}_j}:\mathcal{N}(\mathfrak{g}_j)/G_j\rightarrow\mathcal{N}((\mathfrak{g}_j)_{\mathbb{C}})/(G_j)_{\mathbb{C}}$ is surjective.
\item[(ii)] The Lie algebra $\mathfrak{g}$ is quasi-split if and only if each summand $\mathfrak{g}_j$ is quasi-split.
\end{itemize}
\end{proposition}

\begin{proof}
For each $j\in\{1,\ldots,k\}$, let $\pi_j:\mathfrak{g}\rightarrow\mathfrak{g}_j$ be the projection map. Note that $\xi\in\mathfrak{g}$ is nilpotent if and only if $\pi_j(\xi)$ is nilpotent in $\mathfrak{g}_j$ for each $j$. It follows that $$\pi:\mathcal{N}(\mathfrak{g})\rightarrow\prod_{j=1}^k\mathcal{N}(\mathfrak{g}_j)$$ $$\xi\mapsto(\pi_j(\xi))_{j=1}^k$$ defines an isomorphism of real varieties. Note that $G=\prod_{j=1}^kG_j$, with the former group acting on $\mathcal{N}(\mathfrak{g})$ and the latter group acting on the product of nilpotent cones. 

One then sees that $\pi$ is $G$-equivariant, so that it descends to a bijection $$\overline{\pi}:\mathcal{N}(\mathfrak{g})/G\rightarrow\prod_{j=1}^k\mathcal{N}(\mathfrak{g}_j)/G_j.$$

Analogous considerations give a second bijection $$\overline{\pi}_{\mathbb{C}}:\mathcal{N}(\mathfrak{g}_{\mathbb{C}})/G_{\mathbb{C}}\rightarrow\prod_{j=1}^k\mathcal{N}((\mathfrak{g}_j)_{\mathbb{C}})/(G_j)_{\mathbb{C}}.$$ Furthermore, we have the commutative diagram \begin{equation}\label{diagram}\xymatrix{
\mathcal{N}(\mathfrak{g})/G \ar[d]^{\varphi_{\mathfrak{g}}} \ar@{}[r]^(.15){}="a"^{\overline{\pi}}^(.75){}="b" \ar "a";"b" & \;\;\;\prod_{j=1}^k\mathcal{N}(\mathfrak{g}_j)/G_j \ar[d]^{\prod_{j=1}^k\varphi_{\mathfrak{g}_j}} \\
\mathcal{N}(\mathfrak{g}_{\mathbb{C}})/G_{\mathbb{C}} \ar@{}[r]^(.195){}="a"^{\overline{\pi}_{\mathbb C}}^(.75){}="b" \ar "a";"b" & \;\;\;\;\;\;\;\;\;\;\;\;\prod_{j=1}^k\mathcal{N}((\mathfrak{g}_j)_{\mathbb{C}})/(G_j)_{\mathbb{C}}}.\end{equation} Hence, $\varphi_{\mathfrak{g}}$ is surjective if and only if $\prod_{j=1}^k\varphi_{\mathfrak{g}_j}$ is so, proving (i).

By Lemma \ref{quasi-split}, proving (ii) will be equivalent to proving that $\Theta_{\text{reg}}(\mathfrak{g}_{\mathbb{C}})$ is in the image of $\varphi_{\mathfrak{g}}$ if and only if $\Theta_{\text{reg}}((\mathfrak{g}_j)_{\mathbb{C}})$ is in the image of $\varphi_{\mathfrak{g}_j}$ for all $j$. Using the diagram \eqref{diagram}, this will follow from our proving that the image of $\Theta_{\text{reg}}(\mathfrak{g}_{\mathbb{C}})$ under $\overline{\pi}_{\mathbb{C}}$ is the $k$-tuple of the regular nilpotent orbits in the $(\mathfrak{g}_j)_{\mathbb{C}}$, namely that \begin{equation}\label{image}
\overline{\pi}_{\mathbb{C}}(\Theta_{\text{reg}}(\mathfrak{g}_{\mathbb{C}}))=(\Theta_{\text{reg}}((\mathfrak{g}_j)_{\mathbb{C}}))_{j=1}^k.\end{equation}To see this, note that $\prod_{j=1}^k\Theta_{\text{reg}}((\mathfrak{g}_j)_{\mathbb{C}})$ is the $G_{\mathbb{C}}=\prod_{j=1}^k(G_j)_{\mathbb{C}}$-orbit of maximal dimension in $\prod_{j=1}^k\mathcal{N}((\mathfrak{g}_j)_{\mathbb{C}})$. This orbit is therefore the image of $\Theta_{\text{reg}}(\mathfrak{g}_{\mathbb{C}})$ under the $G_{\mathbb{C}}$-equivariant variety isomorphism $\mathcal{N}(\mathfrak{g}_{\mathbb{C}})\cong\prod_{j=1}^k\mathcal{N}((\mathfrak{g}_j)_{\mathbb{C}})$, implying that \eqref{image} holds.      
\end{proof}

In light of Proposition \ref{reduction}, we address ourselves to classifying the simple real Lie algebras $\mathfrak{g}$ with surjective orbit complexification maps $\varphi_{\mathfrak{g}}$. Noting Lemma \ref{quasi-split}, we may assume $\mathfrak{g}$ to be quasi-split. Since $\mathfrak{g}$ being split is a sufficient condition for surjectivity, we are further reduced to finding those quasi-split simple $\mathfrak{g}$ which are non-split but have surjective $\varphi_{\mathfrak{g}}$. It follows that $\mathfrak{g}$ belongs to one of the four families $\mathfrak{su}(n,n)$, $\mathfrak{su}(n+1,n)$, $\mathfrak{so}(2n+2,2n)$, and $\mathfrak{so}(2n+1,2n-1)$, or that $\mathfrak{g}=EII$, the non-split, quasi-split real form of $E_6$ (see Appendix C3 of \cite{Knapp}). Our examples establish that $\varphi_{\mathfrak{g}}$ is surjective for $\mathfrak{g}=\mathfrak{su}(n,n)$, $\mathfrak{g}=\mathfrak{su}(n+1,n)$, and $\mathfrak{g}=\mathfrak{so}(2n+2,2n)$, while Example \ref{Failed Surjectivity: Indefinite Special Orthogonal} demonstrates that surjectivity does not hold for $\mathfrak{g}=\mathfrak{so}(2n+1,2n-1)$. Also, a brief examination of the computations in \cite{Djokovic2} reveals that $\varphi_{\mathfrak{g}}$ is surjective for $\mathfrak{g}=EII$. We then have the following characterization of the surjectivity condition.

\begin{theorem}\label{Completion of Proof}
If $\mathfrak{g}$ is a semisimple real Lie algebra, then $\varphi_{\mathfrak{g}}$ is surjective if and only if $\mathfrak{g}$ is quasi-split and has no simple summand of the form $\mathfrak{so}(2n+1,2n-1)$.
\end{theorem}

\begin{proof}
If $\varphi_{\mathfrak{g}}$ is surjective, then Lemma \ref{quasi-split} implies that $\mathfrak{g}$ is quasi-split. Also, Proposition \ref{decomposition} implies that each simple summand of $\mathfrak{g}$ has a surjective orbit complexification map, and the above discussion then establishes that $\mathfrak{g}$ has no simple summand of the form $\mathfrak{so}(2n+1,2n-1)$. Conversely, assume that $\mathfrak{g}$ is quasi-split and has no simple summand of the form $\mathfrak{so}(2n+1,2n-1)$. By Proposition \ref{decomposition} (ii), each simple summand of $\mathfrak{g}$ is quasi-split. Furthermore, the above discussion implies that the only quasi-split simple real Lie algebras with non-surjective orbit complexification maps are those of the form $\mathfrak{s0}(2n+1,2n-1)$. Hence, each simple summand of $\mathfrak{g}$ has a surjective orbit complexification map, and Proposition \ref{decomposition} (i) implies that $\varphi_{\mathfrak{g}}$ is surjective. 
\end{proof} 

\subsection{The Image of $\varphi_{\mathfrak{g}}$}\label{The Image}
Having investigated the surjectivity of $\varphi_{\mathfrak{g}}$, let us consider the more subtle matter of characterizing its image. Accordingly, let $\sigma_{\mathfrak{g}}:\mathfrak{g}_{\mathbb{C}}\rightarrow\mathfrak{g}_{\mathbb{C}}$ denote complex conjugation with respect to the real form $\mathfrak{g}\subseteq\mathfrak{g}_{\mathbb{C}}$. The following lemma will be useful.

\begin{lemma}
If $\Theta\subseteq\mathfrak{g}_{\mathbb{C}}$ is a complex nilpotent orbit, then so is $\sigma_{\mathfrak{g}}(\Theta)$.
\end{lemma} 

\begin{proof}
Note that $\sigma_{\mathfrak{g}}$ integrates to a real Lie group automorphism $$\tau:(G_{\mathbb{C}})_{\text{sc}}\rightarrow(G_{\mathbb{C}})_{\text{sc}},$$ where $(G_{\mathbb{C}})_{\text{sc}}$ is the connected, simply-connected Lie group with Lie algebra $\mathfrak{g}_{\mathbb{C}}$. If $g\in (G_{\mathbb{C}})_{\text{sc}}$ and $\xi\in\mathfrak{g}_{\mathbb{C}}$, then $$\sigma_{\mathfrak{g}}(\Adj(g)(\xi))=\Adj(\tau(g))(\sigma_{\mathfrak{g}}(\xi)).$$ Hence, $\sigma_{\mathfrak{g}}$ sends the $(G_{\mathbb{C}})_{\text{sc}}$-orbit of $\xi$ to the $(G_{\mathbb{C}})_{\text{sc}}$-orbit of $\sigma_{\mathfrak{g}}(\xi)$. To complete the proof, we need only observe that $(G_{\mathbb{C}})_{\text{sc}}$-orbits coincide with $G_{\mathbb{C}}$-orbits in $\mathfrak{g}_{\mathbb{C}}$, and that $\sigma_{\mathfrak{g}}(\xi)$ is nilpotent whenever $\xi$ is nilpotent.
\end{proof}

We may now use $\sigma_{\mathfrak{g}}$ to explicitly describe the image of $\varphi_{\mathfrak{g}}$ when $\mathfrak{g}$ is quasi-split.

\begin{theorem}\label{Image Characterization}
If $\Theta$ is a complex nilpotent orbit, the condition $\sigma_{\mathfrak{g}}(\Theta)=\Theta$ is necessary for $\Theta$ to be in the image of $\varphi_{\mathfrak{g}}$. If $\mathfrak{g}$ is quasi-split, then this condition is also sufficient.
\end{theorem} 

\begin{proof}
Assume that $\Theta$ belongs to the image of $\varphi_{\mathfrak{g}}$, so that there exists $\xi\in\Theta\cap\mathfrak{g}$. Note that $\sigma_{\mathfrak{g}}(\Theta)$ is then the complex nilpotent orbit containing $\sigma_{\mathfrak{g}}(\xi)=\xi$, meaning that $\sigma_{\mathfrak{g}}(\Theta)=\Theta$. Conversely, assume that $\mathfrak{g}$ is quasi-split and that $\sigma_{\mathfrak{g}}(\Theta)=\Theta$. The latter means precisely that $\Theta$ is defined over $\mathbb{R}$ with respect to the real structure on $\mathfrak{g}_{\mathbb{C}}$ induced by the inclusion $\mathfrak{g}\subseteq\mathfrak{g}_{\mathbb{C}}$. Theorem 4.2 of \cite{Kottwitz} then implies that $\Theta\cap\mathfrak{g}\neq\emptyset$.
\end{proof}

Using Theorem \ref{Image Characterization}, we will give an interesting sufficient condition for a complex nilpotent orbit to be in the image of $\varphi_{\mathfrak{g}}$ when $\mathfrak{g}$ is quasi-split. In order to proceed, however, we will need a better understanding of the way in which $\sigma_{\mathfrak{g}}$ permutes complex nilpotent orbits. To this end, we have the following lemma.

\begin{lemma}\label{Permutes within Partition Class}
Suppose that $\mathfrak{g}$ comes with the faithful representation $\mathfrak{g}\subseteq\mathfrak{gl}(V)$, where $V$ is over $\mathbb{R}$. If $\Theta$ is a complex nilpotent orbit, then $\lambda(\sigma_{\mathfrak{g}}(\Theta))=\lambda(\Theta)$.
\end{lemma}

\begin{proof}
Choose an $\mathfrak{sl}_2(\mathbb{C})$-triple $(\xi,h,\eta)$ in $\mathfrak{g}_{\mathbb{C}}$ with $\xi\in\Theta$. Since $\sigma_{\mathfrak{g}}$ preserves Lie brackets, it follows that $(\sigma_{\mathfrak{g}}(\xi),\sigma_{\mathfrak{g}}(h),\sigma_{\mathfrak{g}}(\eta))$ is also an $\mathfrak{sl}_2(\mathbb{C})$-triple. The exercise is then to show that our two $\mathfrak{sl}_2(\mathbb{C})$-triples give isomorphic representations of $\mathfrak{sl}_2(\mathbb{C})$ on $\tilde{V}=V_{\mathbb{C}}$. For this, it will suffice to prove that $h$ and $\sigma_{\mathfrak{g}}(h)$ act on $V_{\mathbb{C}}$ with the same eigenvalues, and that their respective eigenspaces for a given eigenvalue are equi-dimensional. To this end, let $\sigma_V:V_{\mathbb{C}}\rightarrow V_{\mathbb{C}}$ be complex conjugation with respect to $V\subseteq V_{\mathbb{C}}$. Note that $$\sigma_{\mathfrak{g}}(h)\cdot(\sigma_V(x))=\sigma_V(h\cdot x)$$ for all $x\in V_{\mathbb{C}}$, where $\cdot$ is used to denote the action of $\mathfrak{g}_{\mathbb{C}}$ on $V_{\mathbb{C}}$. Hence, if $x$ is an eigenvector of $h$ with eigenvalue $\lambda\in\mathbb{R}$, then $\sigma_V(x)$ is an eigenvector of $\sigma_{\mathfrak{g}}(h)$ with eigenvalue $\lambda$. We conclude that $h$ and $\sigma_{\mathfrak{g}}(h)$ have the same eigenvalues. Furthermore, their respective eigenspaces for a fixed eigenvalue are related by $\sigma_V$, and so are equi-dimensional. 
\end{proof}

We now have the following

\begin{proposition}\label{Unique Partition}
Let $\mathfrak{g}$ be a quasi-split semisimple real Lie algebra endowed with a faithful representation $\mathfrak{g}\subseteq\mathfrak{gl}(V)$, where $V$ is over $\mathbb{R}$. If $\Theta$ is the unique complex nilpotent orbit with partition $\lambda(\Theta)$, then $\Theta$ is in the image of $\varphi_{\mathfrak{g}}$.
\end{proposition}

\begin{proof}
By Lemma \ref{Permutes within Partition Class}, $\sigma_{\mathfrak{g}}(\Theta)$ is a complex nilpotent orbit with partition $\lambda(\Theta)$, and our hypothesis on $\Theta$ gives $\sigma_{\mathfrak{g}}(\Theta)=\Theta$. Theorem \ref{Image Characterization} then implies that $\Theta$ is in the image of $\varphi_{\mathfrak{g}}$.
\end{proof}

A few remarks are in order.

\begin{remark}
One can use Proposition \ref{Unique Partition} to investigate whether $\varphi_{\mathfrak{g}}$ is surjective without appealing to the partition-type description of $\varphi_{\mathfrak{g}}$ discussed in Section \ref{The Comlexification Map}. For instance, suppose that $\mathfrak{g}=\mathfrak{so}(2n+2,2n)$, a quasi-split real form of $\mathfrak{g}_{\mathbb{C}}=\mathfrak{so}_{4n+2}(\mathbb{C})$. We refer the reader to Example \ref{Complex Special Orthogonal} for the precise assignment of partitions to nilpotent orbits in $\mathfrak{so}_{4n+2}(\mathbb{C})$. In particular, note that a complex nilpotent orbit is the unique one with its partition if and only if the partition does not have all even parts. Furthermore, as discussed in Example \ref{Surjectivity: Indefinite Special Orthogonal}, there do not exist partitions of $4n+2$ having only even parts such that each part appears with even multiplicity. Hence, each complex nilpotent orbit is specified by its partition, so Proposition \ref{Unique Partition} implies that $\varphi_{\mathfrak{g}}$ is surjective. 
\end{remark}

\begin{remark}
The converse of Proposition \ref{Unique Partition} does not hold. Indeed, suppose that $\mathfrak{g}=\mathfrak{so}(2n,2n)$, the split real form of $\mathfrak{g}_{\mathbb{C}}=\mathfrak{so}_{4n}(\mathbb{C})$.
Recalling Example \ref{Complex Special Orthogonal}, every partition of $4n$ with only even parts, each appearing with even multiplicity, is the partition of two distinct complex nilpotent orbits. Yet, Lemma \ref{split} implies that $\varphi_{\mathfrak{g}}$ is surjective, so that these orbits are in the image of $\varphi_{\mathfrak{g}}$.     
\end{remark}

\subsection{Fibres}\label{Fibres}
In this section, we investigate the fibres of the orbit complexification map $\varphi_{\mathfrak{g}}:\mathcal{N}(\mathfrak{g})/G\rightarrow\mathcal{N}(\mathfrak{g}_{\mathbb{C}})/G_{\mathbb{C}}$. In order to proceed, it will be necessary to recall some aspects of the Kostant-Sekiguchi Correspondence. To this end, fix a Cartan involution $\theta:\mathfrak{g}\rightarrow\mathfrak{g}$. Letting $\mathfrak{k}$ and $\mathfrak{p}$ denote the $1$ and $(-1)$-eigenspaces of $\theta$, respectively, we obtain the internal direct sum decomposition $$\mathfrak{g} = \mathfrak{k}\oplus\mathfrak{p}.$$
This gives a second decomposition $$\mathfrak{g}_{\mathbb{C}}=\mathfrak{k}_{\mathbb{C}}\oplus\mathfrak{p}_{\mathbb{C}},$$ where $\mathfrak{k}_{\mathbb{C}}$ and $\mathfrak{p}_{\mathbb{C}}$ are the
complexifications of $\mathfrak{k}$ and $\mathfrak{p}$, respectively. Let $K\subseteq G$ and $K_{\mathbb{C}}\subseteq G_{\mathbb{C}}$ be the connected closed subgroups with respective Lie algebras $\mathfrak{k}$ and $\mathfrak{k}_{\mathbb{C}}$. The Kostant-Sekiguchi Correspondence is one between the nilpotent orbits in $\mathfrak{g}$ and the $K_{\mathbb{C}}$-orbits in the ($K_{\mathbb{C}}$-invariant) subvariety $\mathfrak{p}_{\mathbb{C}}\cap\mathcal{N}(\mathfrak{g}_{\mathbb{C}})$of $\mathfrak{g}_{\mathbb{C}}$.

\begin{theorem}[The Kostant-Sekiguchi Correspondence]\label{The Correspondence}
There is a bijective correspondence $$\mathcal{N}(\mathfrak{g})/G\rightarrow(\mathfrak{p}_{\mathbb{C}}\cap\mathcal{N}(\mathfrak{g}_{\mathbb{C}}))/K_{\mathbb{C}}$$ $$\mathcal{O}\mapsto\mathcal{O}^{\vee}$$ with the following properties.
\begin{itemize}
\item[(i)] It is an isomorphism of posets, where $(\mathfrak{p}_{\mathbb{C}}\cap\mathcal{N}(\mathfrak{g}_{\mathbb{C}}))/K_{\mathbb{C}}$ is endowed with the closure order \eqref{Closure Order}.
\item[(ii)] If $\mathcal{O}$ is a real nilpotent orbit, then $\mathcal{O}$ and $\mathcal{O}^{\vee}$ are $K$-equivariantly diffeomorphic.
\end{itemize}
\end{theorem}

The first property was established by Barbasch and Sepanski in \cite{Barbasch-Sepanski}, while the second was proved by Vergne in \cite{Vergne}. Each paper makes extensive use of Kronheimer's desciption of nilpotent orbits from \cite{Kronheimer}.

We now prove two preliminary results, the first of which is a direct consequence of the Kostant-Sekiguchi Correspondence.

\begin{lemma}\label{maximal dimension}
If $\mathcal{O}$ is a real nilpotent orbit, then $\mathcal{O}$ is the unique $G$-orbit of maximal dimension in $\overline{\mathcal{O}}$.
\end{lemma}

\begin{proof}
Suppose that $\mathcal{O}'\neq\mathcal{O}$ is another $G$-orbit lying in $\overline{\mathcal{O}}$. By Property (i) in Theorem \ref{The Correspondence}, it follows that $(\mathcal{O}')^{\vee}$ is an orbit in $\overline{(\mathcal{O}^{\vee})}$ different from $\mathcal{O}^{\vee}$. However, $\mathcal{O}^{\vee}$ is an orbit of the complex algebraic group $K_{\mathbb{C}}$ under an algebraic action, and therefore is the unique orbit of maximal dimension in its closure. Hence, $\dim_{\mathbb{R}}((\mathcal{O}')^{\vee})<\dim_{\mathbb{R}}(\mathcal{O}^{\vee})$. Property (ii) of Theorem \ref{The Correspondence} implies that the Kostant-Sekiguchi Correspondence preserves real dimensions, so that $\dim_{\mathbb{R}}(\mathcal{O}')<\dim_{\mathbb{R}}(\mathcal{O})$. 
\end{proof}

We will also require some understanding of the relationship between the $G$-centralizer of $\xi\in\mathfrak{g}$ and the $G_{\mathbb{C}}$-centralizer of $\xi$, viewed as an element of $\mathfrak{g}_{\mathbb{C}}$. Denoting these centralizers
by $G_{\xi}$ and $(G_{\mathbb{C}})_{\xi}$, respectively, we have the following lemma.

\begin{lemma}\label{centralizer}
If $\xi\in\mathfrak{g}$, then $G_{\xi}$ is a real form of $(G_{\mathbb{C}})_{\xi}$.
\end{lemma}

\begin{proof}
We are claiming that the Lie algebra of $(G_{\mathbb{C}})_{\xi}$ is the complexification of the Lie algebra of $G_{\xi}$. The former is $(\mathfrak{g}_{\mathbb{C}})_{\xi}=\{\eta\in\mathfrak{g}_{\mathbb{C}}:[\eta,\xi]=0\}$, while the Lie algebra of $G_{\xi}$ is $\mathfrak{g}_{\xi}=\{\eta\in\mathfrak{g}:[\eta,\xi]=0\}$. If $\eta=\eta_1+i\eta_2\in\mathfrak{g}_{\mathbb{C}}$ with $\eta_1,\eta_2\in\mathfrak{g}$, then $[\eta,\xi]=[\eta_1,\xi]+i[\eta_2,\xi]$. So, $\eta\in(\mathfrak{g}_{\mathbb{C}})_{\xi}$ if and only if $\eta_1,\eta_2\in\mathfrak{g}_{\xi}$. This is equivalent to the condition that $\eta\in(\mathfrak{g}_{\xi})_{\mathbb{C}}\subseteq\mathfrak{g}_{\mathbb{C}}$, so that $(\mathfrak{g}_{\mathbb{C}})_{\xi}=(\mathfrak{g}_{\xi})_{\mathbb{C}}$.
\end{proof}

We may now prove the main result of this section.

\begin{theorem}
If $\mathcal{O}_1$ and $\mathcal{O}_2$ are real nilpotent orbits with the property that $(\mathcal{O}_1)_{\mathbb{C}}=(\mathcal{O}_2)_{\mathbb{C}}$, then either $\mathcal{O}_1=\mathcal{O}_2$ or $\mathcal{O}_1$ and $\mathcal{O}_2$ are incomparable in the closure order. In other words, each fibre of $\varphi_{\mathfrak{g}}$ consists of pairwise incomparable nilpotent orbits.
\end{theorem}

\begin{proof}
Assume that $\mathcal{O}_1$ and $\mathcal{O}_2$ are comparable. Without the loss of generality, $\mathcal{O}_1\subseteq\overline{\mathcal{O}_2}$. We will prove that $\mathcal{O}_1=\mathcal{O}_2$, which by Lemma \ref{maximal dimension} will amount to showing that the dimensions of $\mathcal{O}_1$ and $\mathcal{O}_2$ agree. To this end, choose points $\xi_1\in\mathcal{O}_1$ and $\xi_2\in\mathcal{O}_2$. Since $(\mathcal{O}_1)_{\mathbb{C}}=(\mathcal{O}_2)_{\mathbb{C}}$, we have $\dim_{\mathbb{C}}((G_{\mathbb{C}})_{\xi_1})=\dim_{\mathbb{C}}((G_{\mathbb{C}})_{\xi_2})$. Using Lemma \ref{centralizer}, this becomes $\dim_{\mathbb{R}}(G_{\xi_1})=\dim_{\mathbb{R}}(G_{\xi_2})$. Hence, the (real) dimensions of $\mathcal{O}_1$ and $\mathcal{O}_2$ coincide.
\end{proof}


\begin{thebibliography}{9}
\bibitem{Barbasch-Sepanski} Barbasch, Dan; Sepanski, Mark R. Closure ordering and the Kostant-Sekiguchi correspondence. Proc.
Amer. Math. Soc. 126 (1998), no. 1, pp. 311-317.
\bibitem{Brieskorn} Brieskorn, E. Singular elements of semi-simple algebraic groups. Actes du Congrès International des Mathématiciens (Nice, 1970), Tome 2, pp. 279-284. Gauthier-Villars, Paris (1971).
\bibitem{Chriss-Ginzburg}  Chriss, Neil; Ginzburg, Victor. Representation Theory and Complex Geometry. Birkhäuser Boston, Inc., Boston, MA (1997). x+495 pp. ISBN: 0-8176-3792-3.
\bibitem{Collingwood-McGovern} Collingwood, David H.; McGovern, William M. Nilpotent Orbits in Semisimple Lie Algebras. Van Nostrand
Reinhold Mathematics Series. Van Nostrand Reinhold Co., New York (1993). xiv+186 pp. ISBN:
0-534-18834-6.
\bibitem{Djokovic1} \DJ okovi\'{c}, Dragomir \v{Z}. Explicit Cayley triples in real forms of G2, F4, and E6. Pacific J. Math. 184 (1998), no. 2, pp. 231-255.
\bibitem{Djokovic3} \DJ okovi\'{c}, Dragomir \v{Z}. The closure diagram for nilpotent orbits of the split real form of E8. Cent. Eur. J. Math. 1 (2003), no. 4, pp. 573-643 (electronic).
\bibitem{Djokovic2} \DJ okovi\'{c}, Dragomir \v{Z}. The closure diagrams for nilpotent orbits of real forms of E6. J. Lie Theory 11 (2001), no. 2, pp. 381-413.
\bibitem{Gerstenhaber} Gerstenhaber, Murray. Dominance over the classical groups. Ann. of Math. (2) 74 (1961), pp. 532-569.
\bibitem{Knapp} Knapp, Anthony W. Lie groups Beyond an Introduction. Second edition. Progress in Mathematics, 140. Birkhäuser Boston, Inc., Boston, MA (2002). xviii+812 pp. ISBN: 0-8176-4259-5
\bibitem{Kottwitz} Kottwitz, Robert E. Rational conjugacy classes in reductive groups. Duke Math. J. 49 (1982), no. 4, pp. 785-806.
\bibitem{Kraft-Procesi}  Kraft, Hanspeter; Procesi, Claudio. On the geometry of conjugacy classes in classical groups. Comment. Math. Helv. 57 (1982), no. 4, pp. 539-602.
\bibitem{Kronheimer} Kronheimer, P. B. Instantons and the geometry of the nilpotent variety. J. Differential Geom. 32 (1990),
no. 2, pp. 473-490.
\bibitem{Rothschild} Rothschild, L. Preiss. Orbits in a real reductive Lie algebra. Trans. Amer. Math. Soc. 168 (1972), pp. 403-421.
\bibitem{Schmid-Vilonen} Schmid, Wilfried; Vilonen, Kari. On the geometry of nilpotent orbits. Sir Michael Atiyah: a great mathematician
of the twentieth century. Asian J. Math. 3 (1999), no. 1, pp. 233-274.
\bibitem{Sekiguchi} Sekiguchi, Jir\={o}. Remarks on real nilpotent orbits of a symmetric pair. J. Math. Soc. Japan 39 (1987), no. 1, pp. 127-138. 
\bibitem{Slodowy} Slodowy, Peter. Simple singularities and simple algebraic groups. Lecture Notes in Mathematics, 815. Springer, Berlin (1980). x+175 pp. ISBN: 3-540-10026-1.
\bibitem{Stanley} Stanley, Richard P. Enumerative Combinatorics. Vol. 2.With a foreword by Gian-Carlo Rota and appendix
1 by Sergey Fomin. Cambridge Studies in Advanced Mathematics, 62. Cambridge University Press,
Cambridge (1999). xii+581 pp. ISBN: 0-521-56069-1; 0-521-78987-7
\bibitem{Vergne} Vergne, Mich\`{e}le. Instantons et correspondance de Kostant-Sekiguchi. C. R. Acad. Sci. Paris Sr. I Math.
320 (1995), no. 8, pp. 901-906.
\end{thebibliography}
\end{document}